\documentclass[11pt, letterpaper, reqno]{amsart}

\usepackage{comment}
\usepackage{bbm}
\usepackage{lineno}
\usepackage{amsmath}
\usepackage{amssymb}
\usepackage{amsfonts}
\usepackage{ dsfont }
\usepackage{cases}
\usepackage{cite}
\usepackage[hyphens]{url}
\usepackage{hyperref}
\usepackage{amsthm}
\usepackage{upgreek}
\usepackage{xcolor}

\numberwithin{equation}{section}
\newtheorem{theorem}{Theorem}
\numberwithin{theorem}{section}
\newtheorem*{theorem*}{Theorem}
\newtheorem{corollary}{Corollary}
\numberwithin{corollary}{section}
\newtheorem{lemma}{Lemma}
\numberwithin{lemma}{section}
\newtheorem{proposition}{Proposition}
\numberwithin{proposition}{section}
\newtheorem{assumption}{Assumption}
\numberwithin{assumption}{section}
\newtheorem{definition}{Definition}
\numberwithin{definition}{section}
\newtheorem{remark}{Remark}
\numberwithin{remark}{section}
\newtheorem{example}{Example}
\numberwithin{example}{section}

\newcommand{\leref}{Lemma~\ref}

\newcommand{\exref}{Example~\ref}
\newcommand{\prref}{Proposition~\ref}
\newcommand{\thref}{Theorem~\ref}
\newcommand{\reref}{Remark~\ref}
\newcommand{\asref}{Assumption~\ref}
\newcommand{\eps}{\varepsilon}

\newcommand{\E}{\mathbb{E}}

\newcommand{\Gr}{{\text{Gr}}}
\newcommand{\R}{\mathbb{R}}
\newcommand{\X}{\mathbb{X}}
\newcommand{\kP}{\mathfrak{P}}
\newcommand{\Equiv}{\Longleftrightarrow}
\newcommand{\cC}{\mathcal{C}}
\newcommand{\cT}{\mathcal{T}}
\newcommand{\kB}{\mathfrak{B}}
\newcommand{\supp}{\text{supp}}

\title[]{Transport plans with domain constraints}
\author[]{Erhan Bayraktar} \thanks{This research was supported in part by the National Science Foundation.} 
\address{Department of Mathematics, University of Michigan}
\email{erhan@umich.edu}
\author[]{Xin Zhang} 
\address{Department of Mathematics, University of Michigan}
\email{zxmars@umich.edu}
\author[]{Zhou Zhou}
\address{School of Mathematics and Statistics, University of Sydney}
\email{zhou.zhou@sydney.edu.au}
\date{\today}

\begin{document}
\maketitle

\begin{abstract}
This paper focuses on martingale optimal transport problems when the martingales are assumed to have bounded quadratic variation.
First, we give a result that characterizes the existence of a probability measure satisfying some convex transport constraints in addition to having given initial and terminal marginals. Several applications are provided: martingale measures with volatility uncertainty, optimal transport with capacity constraints, and Skorokhod embedding with bounded times. Next, we extend this result to multi-marginal constraints.
Finally, we consider an optimal transport problem with constraints and obtain its Kantorovich duality. A corollary of this result is a monotonicity principle which gives a geometric way of identifying the optimizer. 
\end{abstract}

\section{Introduction}

Martingale optimal transport has been an active research area in the past decade due to its applications in robust hedging problems in Mathematical Finance. In this set-up one is only given vanilla option prices at certain maturities, which thanks to a result by \cite{RePEc:ucp:jnlbus:v:51:y:1978:i:4:p:621-51} corresponds to fixing the marginals of the martingale measures at these maturities, and tries to obtain model independent no-arbitrage price bounds.
Mathematically, given two probability measures $\alpha, \beta$ on $\mathbb{R}^d$ and a cost function $c$ on $\mathbb{R}^d \times \mathbb{R}^d$, one wants to minimize $\E^P[c(X,Y)]$ among all joint distributions $P$ on $\mathbb{R}^d \times \mathbb{R}^d$ such that $P$ has initial marginal $\alpha$, terminal marginal $\beta$ and $\E^P[Y|X]=X$. However, it is not clear whether there exists such a $P$ satisfying both the marginal and martingale constraints. This question was answered by Strassen \cite{Strassen}: assume $\alpha$ and $\beta$ have finite first moments,
\begin{equation*}
\begin{aligned}
&\exists\,P\text{ s.t. }P\circ X^{-1}=\alpha; \ P\circ Y^{-1}=\beta; \ \E^P[Y|X]=X\  \\
& \Equiv\   \alpha(f)\leq\beta(f),\ \forall\text{ convex functions } f.
\end{aligned}
\end{equation*}
For martingale optimal transport and its application in Mathematical Finance, we refer readers e.g. to \cite{MR3066985},\cite{MR3161649},\cite{MR3256817},\cite{MR3456332},\cite{MR3719009}, and the references therein.

Another strand of literature considered pricing and hedging problems under volatility uncertainty (volatility is not known but is assumed to belong to a bounded interval): \cite{doi:10.1080/13504869500000005,doi:10.1080/13504869500000007,MR2012820,MR1768218}. This is also related to the notion of G-expectations; see e.g. \cite{MR2397805,MR3114916,Peng}. However, in the volatility uncertainty literature, only the underlying stock price process is assumed to be observable and no liquid option prices are given as in the martingale optimal transport problem described above.

In this article our aim is to combine these two different ideas of model uncertainty and analyze the martingale transport problem with bounded volatility. Another motivating factor for us is the fact that  without the volatility restriction, the hedging prices obtained from the martingale optimal transport are all the same for large classes of European, American, Asian, Bermudan options with similar forms of payoff functions (as observed in  \cite{MR3763082} and proved in \cite{MR4047983}), which is of course not financially realistic. On the other hand, there are results indicating that once we have the bounded volatility restriction, these prices are generally not equal (see e.g.,  \cite{2016arXiv160405517A} and \cite{zz4}), which is more practically viable. 

First, we determine when there exists such a martingale measure satisfying the given volatility constraints and the marginals. 
Using \cite[Theorem 7]{Strassen} together with a measurable selection argument, we obtain \prref{p1}. Based on this proposition, we prove a general result \thref{t1}. After giving a financial interpretation of this theorem (see \reref{rmk:finan}), we provide several examples: 1) martingale measures with volatility uncertainty, see subsection~\ref{subsec:volatility}; 2) optimal transport with capacity constraints, see subsection~\ref{subsec:capacity}; 3) Skorokhod embedding with bounded times,  see subsection~\ref{subsec:skorokhod}. 

Subsequently, we extend \thref{t1} to the case of finitely many marginals using a pasting argument; see \thref{t2}. By taking weak limits, we obtain the corresponding in continuous time when all one-dimensional marginals are given, which characterizes the existence of peacocks under constraints; see \thref{t3} and \reref{kellerer}. We also provide examples 
concerning the existence of martingale measures with volatility uncertainty in the case of finitely many marginals and one-dimensional marginals; see \exref{ex:mar1} and \exref{ex:mar2}. 

Finally, we consider the optimization problem \eqref{eo1}, and obtain a duality result. It is a natural generalization of  \cite[Theorem 9.5]{MR3706606} in our setup. Using the duality result established, we prove a general monotonicity principle which characterizes the geometric structure of the optimizer.

The rest of the paper is organized as follows. In the next section, we establish the existence result when there are only two marginals given. In Section 3, we obtain the result when there are finitely many marginals given. In Section 4 we have the result with all the 1-D marginals in continuous time. In Section 5, we obtain the Kantorovich duality. Finally, in Section 6, we deduce the monotonicity principle from the duality result.

\section{Result with two marginals}

We will let $\Omega$ be one of the following three spaces:
\begin{itemize}
\item $\X^{N+1}$, where $\X$ is a polish space, and $N\in\mathbb{N}$;
\item $C[0,1]$, the space of continuous functions $f: [0,1]\mapsto\X$, endowed with the uniform distance metric, where $\X\subset\R^d$ is connected and closed;
\item $D[0,1]$, the space of RCLL functions $g: [0,1]\mapsto\X$, endowed with the Skorokhod metric, where $\X\subset\R^d$ is connected and closed. 
\end{itemize}
Let $T=N$ if $\Omega=\X^{N+1}$ and $T=1$ if $\Omega=C[0,1],D[0,1]$. For any probability measure $P$ and random variable $Y$, $\E^P[Y]:=\E^P[Y^+]-\E^P[Y^-]$ with the convention $\infty-\infty=-\infty$. 

The spaces of probability measures in this paper are endowed with the relativized weak* topology (see e.g. \cite[Appendix 6]{MR3859905}, \cite[Section 6]{Strassen}) as we describe next.
 
Let $G$ and $H$ be continuous functions on $\X$ that are positive and bounded away from $0$.
For $F=G,H$, let
$$\kP_F :=\{\mu\in\kP(\X):\ \mu(F)<\infty\}\ (\text{simply} \ \kP \ \text{if} \ F=1),$$
and $\mathfrak{C}_F$ (simply $\mathfrak{C}$ if $F=1$) be the Banach space of continuous functions $f$ on $\X$ such that
$$\sup_{x\in\X}\frac{|f(x)|}{F(x)}<\infty.$$ 
Define $J:=G\oplus H$ the continuous function on $\X^2$,
$$J(x_0,x_1):=G\oplus H(x_0,x_1):=G(x_0)+H(x_1).$$
Let
$$\kP_J:=\{\mu\in\kP(\X^2):\ \mu(J)<\infty\},$$
and $\mathfrak{C}_{J}$ be set of continuous functions $f$ on $\X^2$ such that
$$\sup_{x\in\X^2}\frac{|f(x)|}{J(x)}<\infty.$$
For $L=G,H,J$, we say a subset of probability measures $\Lambda\subset\kP_L$ is $L$-closed, if for any $(P_n)\subset\Lambda$ and $P$ with
\begin{equation}\label{ee3}
\E^{P_n}[l]\rightarrow\E^{P}[l],\quad\forall l\in \mathfrak{C}_L,
\end{equation}
 we have $P\in\Lambda$. That is, we will endow spaces of probability measures with the topology generated by \eqref{ee3}. When no such $L$ is specified (e.g., we simply say a probability set is closed or weakly closed), then by default we endow the underlying space of probabiliy measures with weak topology, i.e., the topology generated by \eqref{ee3} with $\mathfrak{C}_L$ being the set of bounded and continuous functions.

Let $X$ be the canonical process on $\Omega$ and $(\mathcal{F}_t)_t$ be the filtration generated by $X$. Let $\Gamma:\,\X\mapsto 2^{\kP(\Omega)}$ be such that $\emptyset\neq\Gamma(x)\subset\kP(\Omega^x)$, where $\kP(\Omega)$ is the set of Borel probability measures on $\Omega$, and
$$\Omega^x:=\{\omega\in\Omega:\ \omega_0=x\}.$$
Here $\Gamma(x)$ represents the set of admissible transport plans given $X_0=x$. We assume that the graph of $\Gamma$,
$$\Gr(\Gamma):=\{(x,P'):\ x\in\X, P'\in\Gamma(x)\}$$
is analytic. Denote
$$C:=\{P\in\kP(\Omega):\ P|_{X_0=\omega_0}\in\Gamma(\omega_0),\ P\text{-a.s.}\ \omega\}.$$

Let $\alpha \in \kP_G$ and $\beta \in \kP_H$ be two probability measures on $\mathbb{X}$. Let
\begin{equation}\label{e5}
A:=\{P\in\kP(\Omega):\ P\circ X_0^{-1}=\alpha\}\quad\text{and}\quad B:=\{P\in\kP(\Omega):\ P\circ X_T^{-1}=\beta\}.
\end{equation}

\begin{remark}\label{r1}
Thanks to the analyticity assumption of $\Gr(\Gamma)$, by the Jankov-von Neumann Theorem (see, e.g., \cite[Proposition 7.49]{Shreve}), there exists a universally measurable selector $P'(\cdot)$ such that $P'(x)\in\Gamma(x)$ for any $x\in\X$. Then $P_0\otimes P'\in C$ for any probability measure $P_0$ on $\X$, where
$$P_0\otimes P'(I):=\int_I P_0(d\omega_0)P'(\omega_0,d\omega),\quad I\in\mathcal{B}(\Omega).$$
In particular, this implies that $A\cap C\neq\emptyset$.
\end{remark}

Our aim is to find a necessary and sufficient condition for $A\cap B\cap C\neq\emptyset$. In particular, $\Gamma$ here is treated as a transport constraint from time $0$ to time $T$, which is different from the marginal constraints. Below is the main result of this section.


\begin{theorem}\label{t1}
Assume $\alpha\in\kP_G$, $\beta\in\kP_H$, and
$$(A\cap C)_{0,T}:=\{P\circ (X_0,X_T)^{-1}:\ P\in A\cap C\}$$
is convex and $J$-closed. Then
\begin{equation}\label{e1}
A\cap B\cap C\neq\emptyset\ \Equiv\ \alpha(f^\Gamma)\leq\beta(f),\ \forall\,f\in \mathfrak{C}_H,
\end{equation}
where $\beta(f): =\int_{\X} f \  \beta(dx)$, and $f^\Gamma(x)$ is defined by
\begin{equation}\label{e7}
f^\Gamma(x):=\inf_{Q\in\Gamma(x)}\E^Q[f(X_T)].
\end{equation}
\end{theorem}

We will prove this result at the end of this section. In the case of $\Omega=\X^2$, we have \prref{p1}, which will be useful in proving \thref{t1}. The proof \prref{p1} essentially follows \cite[Theorem 7]{Strassen} together with a measurable selection argument, and we provide it in the appendix for completeness. 

\begin{proposition}\label{p1}
Assume $\Omega=\X^2$, $\alpha\in\kP_G$ and $\beta\in\kP_H$. Moreover let $A\cap C$ be convex and $J$-closed. Then
$$A\cap B\cap C\neq\emptyset\ \Equiv\ \alpha(f^\Gamma)\leq\beta(f),\ \forall\,f\in \mathfrak{C}_H.$$
\end{proposition}

Let us discuss our assumptions in the following remarks. In what follows, we will give natural examples where these assumptions are satisfied.

\begin{remark}\label{r2}
The closedness of $A\cap C$ cannot imply the closedness of $(A\cap C)_{0,T}$. For instance, let $\Omega=\R^3$, $\alpha=\delta_0$,
$$\Gamma(x)=\{P\in\kP(\Omega^x):\ P\circ(X_1,X_2)^{-1}=\delta_{(x_1,x_2)}, \ (x_1,x_2)\in S\},$$
where $S=\{(x_1,x_2):\ x_1>0,x_2>0,x_1x_2\geq 1\}$. Then $A\cap C$ is weakly closed, but $(A\cap C)_{0,2}=\{\delta_0\otimes\delta_x:\ x>0\}$ is not.

Moreover, in the above theorem the assumption $(A\cap C)_{0,T}$ being closed cannot be replaced by $A\cap C$ being closed. Consider again the above with $\beta=\delta_0$. Then obviously $A\cap B\cap C=\emptyset$. However, for any continuous function $f$,
$$\alpha(f^\Gamma)=f^\Gamma(0)=\inf_{(x_1,x_2)\in S}f(x_2)\leq f(0)=\beta(f).$$
\end{remark}

\begin{remark}[\textbf{Financial interpretation}]\label{rmk:finan}
Suppose $\Gamma$ contains the martingale constraint, i.e.,
$$\Gamma(x)\subset\{Q\in\kP(\Omega^x):\ Q\text{ martingale measure}\},\quad x\in\X.$$
Suppose $X$ represents the stock price, and $f$ is the payoff of an option written on $X_T$. Assume $\alpha=\delta_x$.

Then $f^\Gamma(x)$ represents the sub-hedging price of the option $f$ given the current stock price $X_0=x$, and $\beta(f)$ is market price of $f$ (which is consistent with the vanilla option prices). Then the right-hand-side of \eqref{e1} means that the sub-hedging price is smaller than the market price. By symmetry, the super-hedging price is larger than the market price. On the other hand, the left-hand-side of \eqref{e1} means there is a measure consistent with the constraints. As a result, both sides of \eqref{e1} represent no arbitrage. For the role martingale optimal transport plays in finance see \cite{MR3066985}.
\end{remark}

The following lemma gives a useful sufficient condition for closedness of $(A \cap C)_{0,T}$. 
\begin{lemma}\label{l1}
Let $G=H=1$, so that the topology generated by $\eqref{ee3}$ is the weak topology in the usual sense. If $A\cap C$ is weakly compact, then $(A\cap C)_{0,T}$ is weakly closed.
\end{lemma}
\begin{proof}
Let $Q_n\in (A\cap C)_{0,T}$ such that $Q_n\xrightarrow{w}Q$ for some $Q\in\kP(\X^2)$. Then there exists $P_n\in A\cap C$ such that
$$P_n\circ(X_0,X_T)^{-1}=Q_n.$$
Since $A\cap C$ is weakly compact, there exists some $P\in A\cap C$ such that $P_n\xrightarrow{w} P$. Obviously $P\circ(X_0,X_T)^{-1}=Q$, and thus $Q\in (A\cap C)_{0,T}$.
\end{proof}

\subsection{Examples of volatility uncertainty}\label{subsec:volatility}
Our starting point is to consider $C$ as the set of martingale measures with volatility uncertainty. With some compact constraints on the volatility, we can show $A\cap C$ is indeed weakly compact and thus weakly closed (\leref{l1}). Here are some examples.

\begin{example}[Volatility uncertainty in one period]
Let $\X=\R^d$ and $\Omega=\X^2$. Assume $\alpha$ has a finite first moment (i.e., $\alpha(|x|)<\infty$), and let
\begin{equation}\label{e3}
\Gamma(x)=\left\{Q\in\kP(\Omega^x):\ \E^Q[X_T]=x,\ Q\{(x,y):\ |y-x|\leq a(x)\}=1\right\},
\end{equation}
where $a(\cdot)$ is a nonnegative, bounded and continuous function on $\X$.  It can be shown that $\Gr(\Gamma)$ is Borel measurable.

\begin{proposition}\label{prop1}
In this example, $A\cap C$ is convex and weakly compact.
\end{proposition}

\begin{proof}
Convexity is obvious. Now for any $\eps>0$, there exists a compact set $K\subset\X$ such that $\alpha(K)\geq 1-\eps$. Then for any $P\in A\cap C$,
$$P(X\in K^\eps)\geq 1-\eps,$$
where
$$K^\eps:=\left\{(x,y):\ x\in K, |y-x|\leq \sup_{z\in\X}{a(z)}\right\}$$
is a compact set in $\X^2$. Therefore, $A\cap C$ is tight and thus relatively compact by Prokhorov's theorem (see e.g. \cite[Theorem 3.5.13]{MR2767184}).

Assume $P_n\in A\cap C$ such that $P_n\overset{w}\rightarrow P$. Then  by the Portmanteau Theorem (see e.g. \cite[Theorem 1.2]{doi:10.1137/1101016}),
$$P(\{(x,y):\ |y-x|\leq a(x)\})\geq\limsup_{n\rightarrow\infty} P_n(\{(x,y):\ |y-x|\leq a(x)\})=1.$$
Now, let us show the martingale property under the limiting measure.
Let $g$ be any continuous and bounded function on $\X$. Define the compact subset $U^{\epsilon}:=\{(x,y) \in \X^2 : d((x,y),K^{\epsilon}) \leq \epsilon\}$. Let $f^{\epsilon}$ be a continuous function on $\X^2$ such that $0\leq f^{\epsilon} \leq 1$, $f^{\epsilon}$ is compactly supported by $U^{\epsilon}$ and $f^{\epsilon}|_{K^{\epsilon}}=1$. Since $|X_1-X_0|\leq\sup_{z\in\X}a(z)<\infty$ $P_n$-a.s. and $P$-a.s., the function $(x,y) \mapsto (y-x)g(x)f^{\epsilon}(x,y)$ is continuous and bounded for any $\epsilon >0$. According to the definition of weak convergence, we have that
$$\E^P[(X_1-X_0)g(X_0)f^{\epsilon}(X_0,X_1)]=\lim_{n\rightarrow\infty}\E^{P_n}[(X_1-X_0)g(X_0)f^{\epsilon}(X_0,X_1)]=0.$$
As the random variable $|(X_1-X_0)g(X_0)|$ is bounded $P$-a.s., we can conclude by the dominated convergence theorem, 
$$\E^P[(X_1-X_0)g(X_0)]=\lim\limits_{\epsilon \to 0}\E^P[(X_1-X_0)g(X_0)f^{\epsilon}(X_0,X_1)]=0. $$
This implies $P$ is a martingale measure. As a result, $P\in A\cap C$, and thus $A\cap C$ is weakly compact.
\end{proof}

With $\Gamma$ defined in \eqref{e3}, it can be shown that for any function $f:\,\X\mapsto\R$,
\begin{equation*}
\begin{aligned}
f^\Gamma(x)&=\cC(f|_{\bar O(x,a(x))})(x) \\
= &\inf\left\{\sum_{i=0}^d\lambda_i f(y_i):\ |y_i-x|\leq a(x), \lambda_i\geq 0, i=0,\dotso,d,\sum_{i=0}^d\lambda_i=1, \sum_{i=0}^d \lambda_i y_i=x \right\},
\end{aligned}
\end{equation*}
where $\cC(f|_{\bar O(x,b)})(x)$ is given by the convex envelope of $f$ restricted on $\bar O(x,b):=\{y\in\X:\ |y-x|\leq b\}$ and then evaluating at $x$.
\end{example}

\begin{example}[Volatility uncertainty in multiple periods]
Let $\X=\R^d$ and $\Omega=\X^{N+1}$, $N\geq 1$. Assume $\alpha$ has a finite first moment, and let
\begin{equation*}
\Gamma(x)=\left\{ Q\in\kP(\Omega^x):
\begin{aligned}
&\ Q\text{ martingale measure}, \\
& \ Q\{ |X_n-X_{n-1}|\leq a_{n-1}(X_{n-1})\}=1,\ n=1,\dotso,N
\end{aligned}
\right\},
\end{equation*}

where $a_{n-1}$ is a nonnegative, bounded and continuous function on $\X$ for $n=1,\dotso,N$.  
\end{example}
\begin{proposition}
In this example $A\cap C$ is convex and weakly compact, and $f^\Gamma$ can be calculated recursively as follows:
$$g_N=f,\quad g_{n-1}(x)=\cC(g_n|_{\bar O(x,a_{n-1}(x))})(x),\ n=1,\dotso,N,\quad f^\Gamma=g_0.$$
\end{proposition}
\begin{proof}
The proof is similar to \prref{prop1}. It only remains to show $\E^P[X_n|X_{n-1}]=X_{n-1}$ for $n=2, \dotso, N$. Let us show that $\E^P[X_2|X_1]=X_1$, and the rest can be proved by induction. Denote by $\alpha_1$ the distribution of $X_1$ under $P$. Since $|X_1-X_0| \leq \max\limits_{z \in \X} a_0(z)< +\infty$, $\alpha_1$ has finite first moment. Replacing $\alpha$ with $\alpha_1$ in the proof of \prref{prop1}, we directly obtain that $\E^P[(X_2-X_1)g(X_1)]=0$ for any bounded continuous function $g$, which implies that $\E^P[X_2|X_1]=X_1$.
\end{proof}

\begin{example}[Volatility uncertainty in continuous time]\label{ex3}
Let $\X=\R^d$ and $\Omega=C[0,1]$. Assume $\alpha$ has a finite first moment, and let
$$\Gamma(x)=\left\{Q\in\kP(\Omega^x):\ Q \text{ martingale measure},\ \frac{d\langle X\rangle_t}{dt}\in\mathbb{D},\  dt\times Q\text{-a.e.}\right\},$$
where $\mathbb{D}\subset\R^{d\times d}$ is some fixed convex and compact set of matrices. In this case, $f^\Gamma$ is the $G$-expectation of $f$ (see \cite{Peng}).

\begin{proposition}
In this example, $A\cap C$ is convex and weakly compact.
\end{proposition}

\begin{proof}
First, we show $A\cap C$ is tight. We have
$$\lim_{L\rightarrow\infty}\sup_{P\in A\cap C}P(|X_0|>L)=\lim_{L\rightarrow\infty}\alpha\left(\{x: |x|>L \}\right)=0.$$
Moreover, for any $s,t\in[0,1]$, since $\mathbb{D}$ is bounded, by the Burkholder-Davis-Gundy inequality (see e.g. \cite[Theorem 3.3.28]{MR1121940}) there exists some constant $K$ independent of $s$ and $t$ such that
\begin{equation}
\sup_{P\in A\cap C}\E^P[|X_t-X_s|^4]\leq\sup_{P\in A\cap C}\E^P\left[\E^P\left[\sup_{s\leq r\leq t}|X_r-X_s|^4\Big|X_s\right]\right]\leq K|t-s|^2.
\end{equation}
By the moment criterion, $A\cap C$ is tight (see e.g. \cite[Problem 2.4.11]{MR1121940}).

Next we show $A\cap C$ is closed. Let $P^n\in A\cap C$ such that $P^n\overset{w}\rightarrow P$. Obviously $P\in A$. Then using almost the same argument as in the proof of \cite[Lemma 3.2]{Nutz3}, we can show that $P\in C$.
\end{proof}

\end{example}

\subsection{Example of capacity constraint}\label{subsec:capacity}
 In \cite{MR3286490}, Korman and McCann studied the optimal transport problem with capacity constraints. Suppose $f$ and $g$ are two probability density functions on $\mathbb{R}^d$, $c$ is a cost function on $\mathbb{R}^d \times \mathbb{R}^d$, and $\bar{h} \in L^{\infty}(\mathbb{R}^d \times \mathbb{R}^d)$ is a capacity constraint. Define $\Gamma^{\bar{h}}(f,g):=\{h \in L^1(\mathbb{R}^d \times \mathbb{R}^d):  \text{ $h$ has $f,g$ as its marginals, and } h \leq \bar{h} \} $. Under the assumption $\Gamma^{\bar{h}}(f,g) \not = \emptyset$, Korman and McCann proved that any optimizer $h_0$ of the problem,
\begin{align*}
\inf_{h \in \Gamma^{\bar{h}}(f,g)} \int c(x,y)h(x,y) dx dy, 
\end{align*}
is geometrically extreme, i.e., $h_0= \mathds{1}_{W} \bar{h}$ for some measurable set $W \subset \mathbb{R}^d \times \mathbb{R}^d$. 

In this subsection, we give one more criterion for weak closedness of $(A \cap C)_{0,T}$. In doing so, we can apply \thref{t1} and describe when this non-emptiness assumption $\Gamma^{\bar{h}}(f,g) \not = \emptyset$ is satisfied. Actually  we can deal with more general capacity constraints.

Let $R:\X\mapsto\kP(\Omega)$ be a transition kernel, and
$$\Gamma(x):=\left\{Q\in\kP(\Omega^x):\ \frac{Q(dy)}{R(x,dy)}\leq a(x,y)\right\},$$
where $a(\cdot,\cdot)\geq 0$ is a bounded and Borel measurable function. For any Borel measurable set $A\in\mathcal{B}(\Omega)$, according to \cite[Lemma 4.6]{MR3705159}, the function $Q \mapsto \mathbb{E}^Q[1_A]$ is Borel measurable. Since the function $x \mapsto \int_\Omega 1_A a(x,y)R(x,dy)$ is also Borel measurable, so is the set
$$L_A:=\left\{(x,Q)\in\mathbb{X}\times\mathfrak{P}(\Omega):\ \mathbb{E}^Q[1_A]\leq\int_\Omega 1_A a(x,y)R(x,dy)\right\}.$$
It can be easily checked that $ \{(x,Q):\ \ Q\in \mathfrak{P}(\Omega^x)\}$ is closed, and hence the set 
$$\mathcal{L}_A:=L_A\cap \{(x,Q):\ \ Q\in \mathfrak{P}(\Omega^x)\}$$
is Borel measurable. Now let $(A_i)_{i=1}^\infty$ be a countable algebra generating $\mathcal{B}(\Omega)$. Then
$$\text{Gr}(\Gamma)=\cap_{i=1}^\infty\mathcal{L}_{A_i}$$
is Borel measurable, and hence analytic. 

\begin{proposition}
In this example, $A\cap C$ is weakly compact, and thus $(A\cap C)_{0,T}$ is weakly closed.
\end{proposition}

\begin{proof}
By the boundedness of $a(\cdot,\cdot)$, the subset of $\kP(\X \times \Omega)$
$$\Lambda :=\{\alpha \times Q:\ Q \text{\ is \ any \ transition \ kernel \ such that \ }Q(\cdot)\in\Gamma(\cdot)\ \alpha\text{-a.s.}\}$$
is relatively compact. If we can show $\Lambda$ is weakly compact, then the subset of $\kP(\Omega)$, $$A \cap C=\{ \bar{P} \circ \pi_2^{-1}: \bar{P} \in \Lambda \text{ \ and \ } \pi_2(x,y):=y, \ \forall \ (x,y) \in \X \times \Omega \},$$
is also weakly compact. 
Take $\alpha \times Q_n \in \Lambda$ such that $\alpha \times Q_n \xrightarrow{w} \bar{P}^*$.
By the definition of $\Gamma(x)$, there exist Borel measurable functions $b_n$ with $0\leq b_n(\cdot,\cdot)\leq a(\cdot,\cdot)$ such that for $(x,y)\in\X\times\Omega$,
\begin{equation*}
b_n(x,y)R(x,dy)=Q_n(x,y).
\end{equation*}

Consider $L^2(\X \times \Omega)$ over the probability space $(\X \times \Omega, \alpha \times R)$. Since $L^2$ is reflexive, the weak* topology and weak topology coincide. Now because $b_n$ is uniformly bounded, by Banach-Alaoglu theorem (see e.g. \cite[Theorem17.4]{MR0394084}), there exists a Borel measurable function $b^*$ on $\X\times\Omega$ such that $b_n \xrightarrow{w} b^*$, i.e., 
for any measurable function $f$ on $\X \times \Omega$ with $\E^{\alpha \times R}|f|^2<\infty$, 
\begin{equation}\label{ee1}
\E^{\alpha \times R}[fb_n]\rightarrow\E^{\alpha \times R}[fb^*].
\end{equation}
In particular, the above holds for bounded and continuous functions $f$, which implies that
$$\alpha \times b_nR=\alpha \times Q_n \xrightarrow{w} \alpha \times b^*R.$$
So we conclude $\alpha \times b^*R=\bar{P}^*$. 

Note that for any bounded, nonnegative, and measurable function $f$,
$$\E^{\alpha \times R}[fb_n]\leq\E^{\alpha \times R}[fa].$$
By \eqref{ee1},
$$\E^{\alpha \times R}[fb^*]\leq\E^{\alpha \times R}[fa].$$
This implies that $b^*\leq a$, $\alpha \times R$-a.s., and thus $\bar{P}^*= \alpha \times b^*R \in  \Lambda$.
\end{proof}

\subsection{Application to Skorokhod embedding with bounded times}\label{subsec:skorokhod}

\thref{t1} and \exref{ex3} provide a necessary and sufficient condition for the existence of a Skorokhod embedding in bounded time. We will rely on a time change argument to make a connection to Skorohod embedding; see e.g. Hobson \cite{MR2762363}. To wit, let $\Omega=C[0,1]$ with $\X=\R$. Let $\alpha,\beta\in\kP(\X)$ with finite first moments and $\sigma>0$ be a constant.  For $u,r>0$, define
$$\mathcal{Q}_{u,r}:=\left\{Q\in\kP(\bar\Omega):\ Q\text{ martingale measure},\ \frac{d\langle \bar X\rangle_t}{dt}\leq u,\ 0\leq t\leq r,\ dt\times Q\text{-a.e.}\right\},$$
where $\bar\Omega:=C_0[0,\infty)$ is the set of continuous paths $[0,\infty) \to\X$ starting from position $0$, and $\bar X$ is the canonical process on $\overline\Omega$. For any function $f \in \mathfrak{C}$ and $u,r>0$, define
$$f^{u,r}(x):=\inf_{Q\in\mathcal{Q}_{u,r}}\E^Q[f(x+\bar X_r)].$$
We have the following.
\begin{proposition}
For Brownian motion $\kB$ with initial distribution $\kB_0\overset{d}=\alpha$, there exists a stopping time $\tau$ such that
$$\tau\leq\sigma\quad\text{and}\quad \kB_\tau\overset{d}=\beta,$$
if and only if for any $f\in\mathfrak{C}$,
$$\alpha(f^{\sigma,1})\leq\beta(f).$$
\end{proposition}

\begin{proof}
``$\Longrightarrow$''. For $f \in \mathfrak{C}$, we have that
$$\beta(f)=\E^W[f(\kB_\tau)]=\E^W[\E^W[f(\kB_\tau)|\mathcal{B}_0]]\geq\alpha(f^{1,\sigma})=\alpha(f^{\sigma,1}),$$
where $W$ is the probability measure associated with the Brownian motion, and the third (in)equality follows from  $\frac{d\langle\bar X_{\cdot\wedge\tau}\rangle_t}{dt}=0$ for $t>\tau$, and the fourth (in)equality follows from a change of the time scale.

``$\Longleftarrow$''.  Take $d=1, \mathbb{D}=[0,\sigma]$ in \exref{ex3}, and $$\Gamma(x):=\left\{Q\in\kP(\bar\Omega^x):\ Q\text{ martingale measure},\ \frac{d\langle \bar X\rangle_t}{dt}\leq \sigma,\ 0\leq t\leq 1,\ dt\times Q\text{-a.e.}\right\},$$
where $\bar\Omega^x$ is the set of continuous paths starting from position $x$. Then we have $f^{\Gamma}(x)=f^{\sigma,1}(x)$. 
Applying \thref{t1} and \exref{ex3}, there exists $Q\in\mathcal{Q}_{\sigma,1}$ such that
$$Q\circ X_0^{-1}=\alpha\quad\text{and}\quad Q\circ X_1^{-1}=\beta.$$
By the Dambis-Dubins-Schwarz theorem (see e.g. \cite[Theorem 3.4.6]{MR1121940}, we can extend $X$ and $Q$ to the time interval $[0,\infty)$ so that the condition of the theorem is satisfied), $\kB_s:=X_{\mathfrak{T}(s)}$ is a Brownian motion w.r.t. the filtration $\mathcal{G}_s:=\overline{\mathcal{F}}_{\mathfrak{T}(s)}$, having the initial distribution $\mathfrak{B}_0\overset{d}=\alpha$, and $X_t=\kB_{\langle X\rangle_t}$,
where $\overline{\mathcal{F}}_t$ is given by $\cap_{\eps>0}\mathcal{F}_{t+\eps}$ completed by $Q$, and
$$\mathfrak{T}_s:=\inf\{t\geq 0:\ \langle X\rangle_t>s\}.$$
In particular,
$$X_1=\kB_{\langle X\rangle_1}\overset{d}=\beta\quad\text{and}\quad\tau:=\langle X\rangle_1\leq\sigma.$$
\end{proof}

\subsection{Proof of \thref{t1}} 

\begin{proof}
``$\Longrightarrow$''. The argument is similar to the one for \prref{p1}.

``$\Longleftarrow$''. Let
$$\Gamma_{0,T}(x):=\{P\circ (X_0, X_T)^{-1}:\ P\in\Gamma(x)\}.$$
Then
$$f^\Gamma(x)=\inf_{Q\in\Gamma_{0,T}(x)}\E^Q[f(Y_1)],$$
where $Y=(Y_0,Y_1):=(X_0,X_T)$ is the canonical process on $\X^2$ (starting from position $x$). By \prref{p1}, there exists $P^*\in\kP(\X^2)$ such that
$$P^*\circ Y_0^{-1}=\alpha,\ P^*\circ Y_1^{-1}=\beta,\ P^*|_{Y_0}\in\Gamma_{0,T}(Y_0),\ P^*\text{-a.s.}.$$
Let
$$P^*=\alpha\otimes Q^*,$$
be the disintegration of $P^*$, where $Q^*(\cdot)$ is Borel measurable. By restricting to a Borel set $L\in\sigma(Y_0)=\mathcal{B}(\X)$ with $P^*\circ Y_0^{-1}(L)=1$, we may without loss of generality assume that $Q^*(x)\in\Gamma(x)$ for all $x\in\X$. Then the set
$$I_1:=\{(x,P,Q):\ x\in\X,P\in\kP(\Omega),Q=Q^*(x)\}$$
 is Borel measurable. Moreover, since $\Gr(\Gamma)$ is analytic, the set
$$I_2:=\{(x,P,Q):\ x\in\X,P\in\Gamma(x),Q=P\circ X_T^{-1}\}$$
is also analytic. Then the set
$$I_1\cap I_2=\{(x,P,Q):\ x\in\X, P\in\Gamma(x),\ P\circ X_T^{-1}=Q=Q^*(x)\}$$
is analytic. By the Jankov-von Neumann Theorem (see e.g., \cite[Proposition 7.49]{Shreve}), there exists a univerally measurable selector $(P',Q'):\ \X\mapsto\kP(\Omega)\times\kP(\X)$ such that
$$P'\in\Gamma(x),\ (P'(x))\circ X_T^{-1}=Q'(x)=Q^*(x).$$
Define
$$\bar P=\alpha\otimes P'.$$
It can be seen that $\bar P\in A\cap B\cap C$.
\end{proof}

\begin{remark}[\textbf{Extension to moment constraints}]
Let $\mathcal{A} \subset \kP_G(\X)$ be convex and $G$-compact, $\mathcal{B} \subset \kP_H(\X)$ be convex and $H$-closed. Define
$$\mathbb{A}:=\{P\in\kP(\Omega):\ P\circ X_0^{-1}\in\mathcal{A}\}\quad\text{and}\quad\mathbb{B}:=\{P\in\kP(\Omega):\ P\circ X_T^{-1}\in\mathcal{B}\}.$$
Using almost the same argument as above, we have the following. Assume $(\mathbb{A}\cap C)_{0,T}$ is convex and $J$-closed. Then
$$\mathbb{A}\cap\mathbb{B}\cap C\neq\emptyset\ \Equiv\ \inf_{\alpha\in\mathcal{A}}\alpha(f^\Gamma)\leq\sup_{\beta\in\mathcal{B}}\beta(f),\ \forall\,f
\in\mathfrak{C}_H.$$

\end{remark}

\section{Result for multiple marginals}

We still use the three cases of $\Omega$ from the last section. Assume $0=t_0<t_1<\dotso<t_n=T$ such that for $i=1,\dotso,n-1$, $t_i\in\{1,\dotso,N-1\}$ if $\Omega=\X^{N+1}$, and $t_i\in[0,1]$ if $\Omega=C[0,1]$ or $D[0,1]$. For $i=0,\dotso,n-1$, let $\Omega_i=\X^{t_{i+1}-t_i+1},C[0,t_{i+1}-t_i],D[0,t_{i+1}-t_i]$, and $\overline\Omega_i=\X^{N-t_i+1},C[0,1-t_i],D[0,1-t_i]$, if $\Omega=\X^{N+1},C[0,1],D[0,1]$ respectively. Let $\Omega_i^x\subset\Omega_i(\cdot)$ be the space of the paths starting from $x\in\X$. Denote $X_{[0,t]}$ the path from time $0$ to time $t$.

Let $\Gamma_i:\,\X\mapsto 2^{\kP(\Omega_i)}$ such that $\emptyset\neq\Gamma_i(x)\subset\kP(\Omega_i^x)$ for any $x\in\X$, and assume $\Gr(\Gamma_i)$ is analytic, $i=0,\dotso,n-1$. Define $P^{t_i,\omega}$ to be the conditional probability of $P$ given $\omega$ up to time $t_i$, i.e., for any Borel measurable function $f$ on $\Omega$,
$$\E^{P^{t_i,\omega}}[f(\omega\otimes_{t_i}\cdot)]=\E^P[f|\mathcal{F}_t](\omega),\quad P\text{-a.s. }\omega,$$
where for $\omega'\in\overline\Omega_i$ such that $\omega_0'=\omega_{t_i}$,
$$(\omega\otimes_{t_i}\omega')_s=
\begin{cases}
\omega_s,& s<t_i,\\
\omega_{s-t_i}',& s\geq t_i.
\end{cases}
$$ 
Let
\begin{equation}\label{e8}
C_i:=\{P\in\kP(\Omega_i):\ P|_{X_0=\omega_0}\in\Gamma_i(\omega_0),\ P\text{-a.s. }\omega\},
\end{equation}
and
\begin{equation}\label{e15}
\overline C_i:=\left\{P\in\kP(\Omega):\ P^{t_i,\omega,}\circ X_{[0,t_{i+1}-t_i]}^{-1}\in\Gamma_i(\omega_{t_i}),\ P\text{-a.s. }\omega\right\},
\end{equation}
where $P^{t_i,\omega,}\circ X_{[0,t_{i+1}-t_i]}^{-1}$ represents the marginal probability distribution of $P^{t_i,\omega}$ from time $0$ to time $t_{i+1}-t_i$. 

\begin{remark}
Here $\Gamma_i$ represents the restriction of probability measures from time $t_i$ to time $t_{i+1}$. Note that the restriction only depends on the current location instead of the whole history (i.e., path). This property is critical for the construction of probability measures with multiple marginals later on. Also note that it does not imply the underlying probability measure is Markovian.
\end{remark}


\begin {example}\label{ex:mar1}
Assume $\Omega=C[0,1]$ with $\X=\R^d$. Let $P\in\kP(\Omega)$ be a martingale measure such that
\begin{equation}\label{e6}
\frac{d\langle X\rangle_t}{dt}\in\mathbb{D},\quad dt\times P\text{-a.e.},
\end{equation}
where $\mathbb{D}\subset\R^{d\times d}$ is some bounded set of matrices. Then this martingale and volatility uncertainty restriction satisfies  the property mentioned above. To be more specific, let
$$\Gamma_i(x):=\left\{Q\in\kP(\Omega_i^x):\ Q\text{ martingale measure},\ \frac{d\langle X\rangle_t}{dt}\in\mathbb{D},\ dt\times Q\text{-a.e.}\right\}.$$
Then $P$ satisfies \eqref{e6} if and only if $P\in\cap_{i=0}^{n-1}\overline C_i$.
\end{example}

Let $\alpha_i\in\kP(\X)$, and
\begin{equation*}
\begin{aligned}
&A_i:=\{P\in\kP(\Omega_i):\ P\circ X_0^{-1}=\alpha_i\}\quad\text{and} \\
& \overline A_i:=\{P\in\kP(\Omega):\ P\circ X_{t_i}^{-1}=\alpha_i\},\quad i=0,\dotso,n.
\end{aligned}
\end{equation*}
Recall $f^\Gamma$ defined in \eqref{e7}. The following is the main result of this section.

\begin{theorem}\label{t2}
Let $G=H$. Assume $\alpha_i \in \kP_H$ and $(A_i\cap C_i)_{0,t_{i+1}-t_i}$ is convex and $J$-closed for $i=0,\dotso,n$. Then
$$\bigcap_{i=0}^n\overline A_i\cap \bigcap_{j=0}^{n-1}\overline C_j\neq\emptyset\ \Equiv\  \alpha_i(f^{\Gamma_i})\leq\alpha_{i+1}(f),\ \forall\,f\in\mathfrak{C}_H,\ i=0,\dotso,n-1.$$
\end{theorem}

\begin{proof}
``$\Longrightarrow$ ''. Take $P\in(\bigcap_{i=0}^n\overline A_i)\cap(\bigcap_{j=0}^{n-1}\overline C_j)$. For $i=0,\dotso,n-1$,
$$\alpha_{i+1}(f)=\E^P[f(X_{t_{i+1}})]=\E^P[\E^P[f(X_{t_{i+1}})|\mathcal{F}_{t_i}]]\geq\E^P[f^{\Gamma_i}(X_{t_i})]=\alpha_i(f^{\Gamma_i}),$$
where the inequality follows from the definition in \eqref{e7}, and the fact that the conditional probability associated with $\E^P[\cdot|\mathcal{F}_{t_i}](\omega)$ is an element of $\Gamma_i(\omega_{t_i})$ for $P$-a.s. $\omega$ (see \eqref{e8} and \eqref{e15}).

``$\Longleftarrow$''. By \thref{t1} there exists a probability measure $P_i\in A_i\cap B_i\cap C_i$ on $\Omega_i$ for $i=0,\dotso,n-1$, where
$$B_i:=\{P\in\kP(\Omega_i):\ P\circ X_{t_{i+1}-t_i}^{-1}=\alpha_{i+1}\}.$$
Let $P:=P_0\otimes\dotso\otimes P_{n-1}$. That is,
$$P(I):=\int_IP_0(d\omega_{[t_0,t_1]})P_1(\omega_{t_1},d\omega_{[t_1,t_2]})\dotso P_{n-1}(\omega_{t_{n-1}},d\omega_{[t_{n-1},t_n]}),\quad I\in\mathcal{B}(\Omega).$$
where for $x\in\X$,
\begin{equation}\label{eq:tranker}
P_i(x,\cdot):=P_i|_{\omega_0=x}
\end{equation}
is the conditional probability of $P_i$ given $\omega_0=x$. It can be shown that $P$ indeed is a probability measure on $\Omega$. Moreover, $P^{t_i,\omega}\circ X_{[0,t_{i+1}-t_i]}^{-1}=P_i(\omega_{t_i},\cdot)\in\Gamma_i(\omega_{t_i})$ for $P$-a.s. $\omega$, and thus $P\in \overline C_i$ for $i=0,\dotso,n-1$.

It remains to show that
\begin{equation}\label{e9}
P\circ X_{t_i}^{-1}=\alpha_i,\quad i=0,\dotso,n.
\end{equation}
We prove the above by induction. Obviously \eqref{e9} holds for $i=0$. Assume it holds for $i=k$ with $0\leq k\leq n-1$, and consider the case when $i=k+1$. For any bounded and measurable function $f$ on $\X$, we have that
\begin{align*}
\E^P[f(X_{t_{k+1}})]&=\E^P[\E^P[f(X_{t_{k+1}})|X_{t_k}]]\\
&=\int_\X\alpha_k(dx)\int_{\Omega_k^x}f(\omega_{t_{k+1}-t_k})P_k(x,d\omega)\\
&=\int_{\Omega_k}f(\omega_{t_{k+1}-t_k})\alpha_k(dx)P_k(x,d\omega)\\
&=\E^{P_k}[f(X_{t_{k+1}-t_k})]\\
&=\alpha_{k+1}(f),
\end{align*}
where the second equality follows from the induction hypothesis $P\circ X_{t_k}^{-1}=\alpha_k$ and \eqref{eq:tranker}, the fourth equality follows from $P_k\in A_k$, and the fifth from $P_k\in B_k$.
\end{proof}

\section{Result with all the 1-D marginals in continuous time}

In this section, we consider two cases $\Omega=C[0,1]$ or $D[0,1]$. For $t\in[0,1]$, let $\Omega_t=C[0,t],D[0,t]$ when $\Omega=C[0,1],D[0,1]$ respectively, $\Omega_t^x\subset\Omega_t$ be the set of paths starting from position $x\in\X$. We are given a class of maps $\Gamma_{[s,t]}:\,\X\mapsto 2^{\kP(\Omega_{t-s})}$ for $0\leq s<t\leq 1$. Each $\Gamma_{[s,t]}$ will represent the restriction of probability measures to the time interval $[s,t]$. In particular, this restriction is Markovian in the sense that $\Gamma_{[s,t]}(\cdot)$ only depends on the current value $\omega_s\in\X$ instead of the whole history $\omega_{[0,s]}$. Again we assume that for any $0\leq s<t\leq 1$, $\emptyset\neq\Gamma_{[s,t]}(x)\subset\kP(\Omega_{t-s}^x)$ for $x\in\X$, and $\Gr(\Gamma_{[s,t]})$ is analytic.

For $0\leq s<t\leq 1$, let
$$C_{[s,t]}:=\{P\in\kP(\Omega):\ P^{s,\omega}\circ X_{[s,t]}^{-1}\in\Gamma_{[s,t]}(\omega_s),\ P\text{-a.s. }\omega\}.$$
We assume $\{\Gamma_{[s,t]}\}_{0\leq s<t\leq 1}$ is such that the following consistency property holds:
\begin{equation}\label{e10}
C_{[s,t]}\cap C_{[s',t']}=C_{[s\wedge s',t\vee t']},\quad\text{if}\ [s,t]\cap[s',t']\neq\emptyset.
\end{equation}

Let $(\alpha_t)_{t\in[0,1]}\subset\kP(\X)$. We will consider probability measures on $\Omega$ with marginals $(\alpha_t)_{t\in[0,1]}$. We assume the map $t\mapsto\alpha_t$ is continuous if $\Omega=C[0,1]$, and is right continuous if $\Omega=D[0,1]$ (otherwise $(\alpha_t)_{t\in[0,1]}$ cannot be the marginals of any $P\in\kP(\Omega)$). Define
$$A_t:=\{P\in\kP(\Omega):\ P\circ X_t^{-1}=\alpha_t\},\quad t\in[0,1].$$

Below is the main result of this section.
\begin{theorem}\label{t3}
Assume $A_s\cap C_{[s,t]}$ is weakly compact for any $0\leq s<t\leq T$. Then
$$\bigcap_{0\leq r\leq 1}A_r\cap\bigcap_{0\leq s<t\leq 1}C_{[s,t]}\neq\emptyset\ \Equiv\ \alpha_s(f^{\Gamma_{[s,t]}})\leq\alpha_t(f),\ \forall\,f\in\mathfrak{C},\ 0\leq s<t\leq 1.$$
\end{theorem}

\begin{proof}
``$\Longrightarrow$'' follows from the same argument used in the proof of \prref{p1}.

``$\Longleftarrow$'' By \thref{t2}, there exists $P^n\in \Lambda^n$, where
$$\Lambda^n:=\bigcap_{i=0}^{2^n} A_{i/2^n}\cap\bigcap_{j=0}^{2^n-1} C_{[j/2^n,(j+1)/2^n]}.$$
According to our assumption $A_s\cap C_{[s,t]}$ is weakly compact for any $0\leq s<t\leq T$, it is easy to show that $\Lambda^n$ is weakly compact for any $n\in\mathbb{N}$. By the consistency assumption \eqref{e10}, it follows that 
$$\bigcap_{j=0}^{2^n-1} C_{[j/2^n,(j+1)/2^n]}=C_{[0,1]},$$ and hence
$$\Lambda^{n+1}=\bigcap_{i=0}^{2^{n+1}} A_{i/2^{n+1}}\cap C_{[0,1]} \subset \bigcap_{i=0}^{2^{n}} A_{i/2^{n}}\cap C_{[0,1]}=\Lambda^n.$$ Therefore, $P^m\in\Lambda^n$ for any $m\geq n$. In particular, $P^m\in\Lambda^1$ with $\Lambda^1$ weakly compact. Then there exists $P\in\kP(\Omega)$ such that
$$P^m\overset{w}\rightarrow P.$$
It can be seen that $P\in\Lambda^n$ for any $n\in\mathbb{N}$.

The proof of  $P\in\bigcap_{0\leq r\leq 1}A_r\cap\bigcap_{0\leq s<t\leq 1}C_{[s,t]}$ goes as follows. By \eqref{e10}, $P\in C_{[0,1]} \subset C_{[s,t]}$ for any $0\leq s<t\leq 1$. 
If $t\in\cT$, where
$$\cT:=\{k/2^n:\ k=0,\dotso,2^n,\ n\in\mathbb{N}\},$$
then $P\circ X_t^{-1}=\alpha_t$, since $P_n\circ X_t^{-1}=\alpha_t$ for $n$ large enough. In general, for $t\in[0,1]$, let $t^k\in\cT$ such that $t^k\searrow t$. Since $X$ is right continuous,
$$\alpha_{t^k}=P\circ X_{t^k}^{-1}\overset{w}\rightarrow P\circ X_t^{-1}.$$
As $t\mapsto\alpha_t$ is right continuous, we have $P\circ X_t^{-1}=\alpha_t$.
\end{proof}

\begin{remark}
The result still holds and the proof still goes through with minor adjustments, if we weaken/replace the assumption by: (1) there exists $\mathfrak{T}\subset[0,1]$ that is dense in $[0,1]$, such that $(A_s\cap C_{[s,t]})\circ(X_s,X_t)^{-1}$ is convex and closed for any $s,t\in\mathfrak{T}$ with $s<t$; (2) $A_s\cap C_{[s,t]}$ is weakly compact for any $s,t\in\mathfrak{T}$ with $s<t$; (3) the consistency assumption \eqref{e10}.
\end{remark}

\begin{example}[Martingale measures with volatility uncertainty]\label{ex:mar2}
Let $\Omega=C[0,1]$ with $\X=\R^d$. Assume $\alpha_t$ has a finite first moment for any $t\in[0,1]$. Let
\begin{equation*}
\Lambda:=\left\{P\in\kP(\Omega):
\begin{aligned}
&\ P\circ X_t^{-1}=\alpha_t,\ t\in[0,1],\ P \text{ martingale measure}, \\
&\ \frac{d\langle X\rangle_t}{dt}\in\mathbb{D},\ dt\times P\text{-a.e.} 
\end{aligned}
\right\},
\end{equation*}
where $\mathbb{D}\subset\R^{d\times d}$ is a convex and compact set of matrices.

Then it can be seen that
$$\Lambda=\bigcap_{0\leq r\leq 1}A_r\cap\bigcap_{0\leq s<t\leq 1}C_{[s,t]},$$
with $\Gamma_{[s,t]}$ defined by
\begin{equation*}
\begin{aligned}
&\Gamma_{t-s}:=\Gamma_{[s,t]}(x) \\
:=& \left\{Q\in\kP(\Omega_{t-s}^x):\ Q\text{ martingale measure},\ \frac{d\langle X\rangle_r}{dr}\in\mathbb{D},\ dr\times Q\text{-a.e.}\right\}.
\end{aligned}
\end{equation*}
Moreover, with $\Gamma$ defined above, the consistency condition \eqref{e10} is obviously satisfied, and $A_s\cap C_{[s,t]}$ is weakly compact for any $0\leq s<t\leq T$. Therefore, by \thref{t3},
$$\Lambda\neq\emptyset\ \Equiv\ \ \alpha_s(f^{\Gamma_{t-s}})\leq\alpha_t(f),\ \forall\,f\in\mathfrak{C},\ 0\leq s<t\leq 1.$$
Again,
$$f^{\Gamma_r}(x)=\inf_{Q\in\Gamma_r(x)}\E^Q[f(X_T)]$$
is the $G$-expectation of $f$ (see \cite{Peng}).
\end{example}

\begin{remark}\label{kellerer}
When $\X=\mathbb{R}^d$, the existence of a martingale measure without volatility constraint with given marginals $(\mu_t)_t$ is characterized by Kellerer in \cite{kellerer}, Hirsch and Roynette in \cite{MR3066388}. For any stochastic process $X$, denote by $\mathcal{F}^X$ the filtration $\mathcal{F}^X(t):=\sigma(X_s, s \leq t)$. Then 
$$ \exists\ \text{martingale $X$ w.r.t. $\mathcal{F}^X$ s.t. }X_t\overset{d}=\mu_t\ \Equiv\ t\mapsto\mu_t(f)\ \text{is increasing$,\ \forall$ convex functions }f.$$
In particular for $d=1$, Kellerer showed that the martingale can be Markov. 
\end{remark}

\section{Kantorovich duality}
In this section, we will provide the Kantorovich duality with our domain constraint as in section 2. Our proof idea is similar to \cite[Theorem 9.5]{MR3706606} where it proved an unconstrained result. Here we use the usual weak topology, but the results can be easily generalized to relativized case. 

Consider the optimization problem 
\begin{equation}\label{eo1}
\mathcal{T}_c^{\Gamma}(\alpha,\beta)=\inf_{\pi \in \Pi_{\Gamma}(\alpha,\beta)} \int_\mathbb{X} c(x,\pi|_{X_0=x}) \alpha(dx),
\end{equation}
where $$\Pi_{\Gamma}(\alpha,\beta):=\{ \pi \in \kP(\Omega): \ \pi \circ X_0^{-1}=\alpha,\ \pi \circ X_T^{-1}=\beta, \  \pi|_{X_0=x} \in \Gamma(x) \ \pi \text{-a.s.} \}$$ is the set of probability measures with marginals $\alpha, \beta$ and domain constraint $\Gamma$. We make the following assumption.
\begin{assumption}\label{a1}
\begin{itemize}
\item[(i)]  The cost function $c:\,\mathbb{X}\times\mathfrak{P}(\Omega) \to [0,\infty]$ is lower-semicontinuous with respect to product topology.
\item[(ii)] The function $Q \mapsto c(x,Q)$ is convex for all $x \in \mathbb{X}$.
\end{itemize}
\end{assumption}

\begin{example}
If $c$ is given by
\begin{equation}\label{eo6}
c(x,Q)=\int _{y \in \mathbb{X}} \cC(x,y) \; Q\circ X_T^{-1}(dy),
\end{equation}
where $\cC: \mathbb{X} \times \mathbb{X} \to [0,\infty]$ is continuous. Then
\begin{equation}\label{eo7}
\mathcal{T}_c^{\Gamma}(\alpha,\beta)=\inf_{\pi \in \Pi_{\Gamma}(\alpha,\beta)} \int \cC(x,y) \; \pi(dx,dy).
\end{equation}
In this case, $c$ is linear with respect to $Q$ and \asref{a1} is satisfied.
\end{example}

\begin{remark}
By a slight modification of \cite[Proposition 2.8]{MR4029731}, it can be seen that the function $$\pi \mapsto I_c[\pi]:=\int c(x,\pi|_{X_0=x}) \alpha(dx)$$ is lower-semicontinuous under \asref{a1}.
\end{remark}

\begin{remark}
Assume $\Omega=\X^2$ and $A\cap C$ is weakly closed. \prref{p1} provides a necessary and sufficient condition for the non-emptiness of the weakly compact set $A \cap B \cap C$. Then under \asref{a1}, the infimum in \eqref{eo1} is attained. 
\end{remark}

We use $\Phi $ (resp. ${\Phi}_b(\mathbb{X})$) to denote the set of continuous (resp. continuous and bounded from below) functions $\phi: \mathbb{X} \to \mathbb{R}$ satisfying the linear growth condition
$$|\phi(x)| \leq a+b \:d(x,x_0), \forall x \in \mathbb{X}, $$ for some $a,b \geq 0$ and some (and hence all) $x_0 \in \mathbb{X}$. Below is the Kantorovich duality with the domain constraint.

\begin{theorem}\label{t4}
Assume $\Omega=\X^2$, $A \cap C$ is convex and weakly closed, and let \asref{a1} hold. Then
\begin{equation}\label{eo2}
\mathcal{T}_c^{\Gamma}(\alpha,\beta)=\sup_{\phi \in \Phi_b(\mathbb{X})} \left\{   \int R_c^{\Gamma} \phi(x) \alpha(dx)-\int \phi(y) \beta(dy)          \right\},
\end{equation}
where $$R_c^{\Gamma} \phi(x):= \inf_{Q \in \Gamma(x)} \{ \mathbb{E}^Q[\phi(X_T)] + c(x,Q) \},\quad x \in \mathbb{X}, \phi \in \Phi_b(\mathbb{X}).$$
\end{theorem}

\begin{proof}
We will apply Fenchel-Moreau theorem (see e.g. \cite[Theorem 4.2.1]{Borwein00convexanalysis}). For the rest of the proof, let $\mathfrak{M}(\mathbb{X})$ be the space of all Borel signed measures with finite first moments. We equip it with weak topology.

Consider $F:\,\mathfrak{M}(\mathbb{X})\mapsto[0,\infty]$ defined by $$F(m)= \mathcal{T}_c^{\Gamma}(\alpha,m)=\inf_{\pi \in \Pi_{\Gamma}(\alpha,m)} \int_\mathbb{X} c(x,\pi|_{X_0=x}) \alpha(dx),$$ with the convention $\inf \emptyset = + \infty.$ As $A\cap C$ is convex and weakly closed, first we show that the set
$$\mathfrak{Im}:=\{P \circ X_T^{-1}: P \in A \cap C\}$$
is also convex and weakly closed. Take any convergent sequence $\{m_n\}_{n \in \mathbb{N}} \subset \mathfrak{Im}$, with $\{\pi^n\}_{n \in \mathbb{N}} \subset A \cap C$ such that $\pi^n \circ X_T^{-1}=m_n$. For any $\epsilon > 0$,  since $\{m_n\}_{n \in \mathbb{N}}$ is relatively compact, we could find a compact set $K_{\epsilon} \subset \mathbb{X}$ such that $m_n(K_{\epsilon}) \ge 1-\epsilon$ for each $n$. Let $L_{\epsilon} \subset \mathbb{X}$ be a compact set such that $\alpha(L_{\epsilon} ) \ge 1-\epsilon$. We get that $\pi^n(L_{\epsilon} \times K_{\epsilon}) \ge 1-2\epsilon$ for each $n$ and therefore conclude $\{\pi^n\}_{n \in \mathbb{N}}$ is relatively compact by Prokhorov's Theorem. By the closedness of $A \cap C$, the limit $\pi$ of the sequence $\{\pi^n\}_{n \in \mathbb{N}}$ (up to a subsequence) is in $A\cap C$. It is clear that $\{m_n\}_{n \in \mathbb{N}}$ converges to $\pi \circ X_T^{-1} \in \mathfrak{Im}$ and we conclude.

Next, we show that $F$ is convex. Take $m_0,m_1 \in \mathfrak{P}(\mathbb{X})$. If either one of $F(m_0)$ and $F(m_1)$ is positive infinity, then we trivially have 
$$F(tm_0+(1-t)m_1) \leq tF(m_0)+(1-t)F(m_1),\quad\forall\,t \in (0,1).$$
Thus we assume $m_0,m_1 \in \mathfrak{Im}$ without loss of generality. Take $\pi^i \in \Pi_{\Gamma}(\alpha,m_i), i=0,1$. Since the cost function $c$ is convex in its second argument, it holds that 
\begin{equation*}
\begin{split}
F(tm_0+(1-t)m_1) & \leq \int c(x,t\pi^0|_{X_0=x}+(1-t)\pi^1|_{X_0=x})  \alpha(dx) \\
                                 & \leq t \int c(x,\pi^0|_{X_0=x}) \alpha(dx) + (1-t) \int c(x,\pi^1|_{X_0=x}) \alpha(dx). \\
\end{split}
\end{equation*}
Optimizing over $\pi^0,\pi^1$, we get that
$$F(tm_0+(1-t)m_1) \leq tF(m_0)+(1-t)F(m_1),\quad\forall t\in(0,1),$$
which implies the convexity of $F$.

Then we prove that $F$ is lower semicontinuous. Let $\{m_n\}_{n \in \mathbb{N}}$ converges to $m$ in the weak topology. If $m \notin \mathfrak{Im}$, then by the closedness of $\mathfrak{Im}$, we have that $m_n\notin\mathfrak{Im}$ for $n$ large enough. This implies that
$$\liminf_{n \to \infty} F(m_n)=+ \infty = F(m).$$
Now consider the case $m \in \mathfrak{Im}$. Without loss of generality, we assume the limit $\lim_{n \to \infty} F(m_n)$ exists and is finite.
Let $\pi^n \in \Pi_{\Gamma}(\alpha,m_n) \subset A \cap C$ such that $$\int c(x,\pi^n|_{X_0=x}) \alpha(dx) \leq F(m_n)+\frac{1}{n}.$$ By the same argument as in the second paragraph, we know $\{\pi^n\}_{n \in \mathbb{N}}$ is relatively compact. Extracting a subsequence, we can assume $\{ \pi^n \}_{n \in \mathbb{N}}$ converges to $\pi$ without loss of generality. It is easily seen that $\pi \in \Pi_{\Gamma}(\alpha,m)$. By \asref{a1},
$$F(m) \leq I_c[\pi] \leq \lim \inf_{n \to \infty} I_c[\pi_n]=\lim_{n \to \infty} F(m_n).$$

Notice that $\mathfrak{M}^*(\mathbb{X})$ can be identified with $\Phi(\mathbb{X})$(see e.g. \cite[Lemma 9.8]{MR3706606}), i.e., for any $l \in \mathfrak{M}^*(\mathbb{X})$, there is one corresponding $\phi \in \Phi(\mathbb{X})$ such that 
$$l(m)=\int_\mathbb{X} \phi(x) m(dx),\quad\forall\,m \in \mathfrak{M}(\mathbb{X}).$$
Therefore Fenchel-Legendre transform $F^*(l):=\sup\limits_{m \in \mathfrak{M}} \{l(m)-F(m)\}$ is equivalent to $F^*(\phi)=\sup\limits_{m \in \mathfrak{M}} \{\int \phi dm -F(m) \}.$
Applying Fenchel-Moreau theorem, we get that 
$$F(m)=\sup_{\phi \in \Phi(\mathbb{X})} \bigg\{ \int \phi \: dm - F^*(\phi) \bigg\}=\sup_{\phi \in \Phi(\mathbb{X})} \bigg\{ \int -\phi \: dm - F^*(-\phi) \bigg\}.$$ 
Replacing $\phi$ by $\phi  \vee k$ and letting $k \to -\infty$, we can restrict the last supremum to $\Phi_b(\mathbb{X})$. 

To conclude the proof, we show that
$$F^*(-\phi)=- \int R_c^{\Gamma} \phi(x) \alpha(dx),\quad \forall \phi \in \Phi_b(\mathbb{X}).$$ Since $F$ is positive infinity outside $\mathfrak{Im}$, we have that
\begin{equation*}
\begin{split}
F^*(-\phi) &=\sup_{m \in \mathfrak{Im}} \left\{\int -\phi \: dm -F(m)\right\}\\ 
            &=\sup_{m \in \mathfrak{Im}} \sup_{\pi \in \Pi_{\Gamma}(\alpha,m)} \left\{ \int -\phi \: dm- I_c[\pi]\right\} \\
          &=-\inf_{\pi \in A \cap C} \left\{ \int [\mathbb{E}^{\pi|_{X_0=x}}[\phi(X_T)]+c(x,\pi|_{X_0=x})] \: \alpha(dx)\right\} \\
          &\leq \int \left(-\inf_{Q \in \Gamma(x)} \{ [\mathbb{E}^{Q}[\phi(X_T)]+c(x,Q)]\} \right) \alpha(dx)\\
            &= - \int R_c^{\Gamma}\phi(x) \alpha(dx). \\
 \end{split}
\end{equation*}
On the other hand, for any $\eps>0$, by \cite[Proposition 7.50]{Shreve} there exists a universally measurable probability kernel $P^\eps: \mathbb{X} \times \mathfrak{P}(\Omega) \to \mathbb{R}$ such that 
$$\mathbb{E}^{P^\eps(x,\cdot)}[\phi(X_T)] + c(x,P^\eps(x,\cdot)) \leq R_c^{\Gamma}\phi(x)+\eps.$$
Therefore,
$$ F^*(-\phi) \geq -\int [\mathbb{E}^{P^\eps(x,\cdot)}[\phi(X_T)]+c(x,P^\eps(x,\cdot))] \alpha(dx) \geq - \int R_c^{\Gamma} \phi(x) \alpha(dx) - \eps.$$
Taking $\eps \to 0$, we conclude the result. 
\end{proof}

\begin{remark}
For $\Omega \not = \X^2$, weak closedness of $A \cap C$ cannot imply closedness of $\mathfrak{Im}$ (see \reref{r2}). But  
if we assume $A \cap C$ is convex and weakly compact, and let \asref{a1} hold, we still have \eqref{eo2} by using the same argument as above.
\end{remark}

\begin{corollary}
Let $\Omega=\mathbb{X}^2$. Assume $A\cap C$ is convex and weakly closed, and $c$ is given by \eqref{eo6}. Then 
$$\mathcal{T}_c^{\Gamma}(\alpha,\beta)=\sup_{(f,g) \in \mathcal{F}^{\Gamma}(\alpha,\beta)} \bigg\{   \int f(x) \; \alpha(dx) + \int g(y) \; \beta(dy)       \bigg\},$$
where 
\begin{equation*}
\mathcal{F}^{\Gamma}(\alpha,\beta) =\left\{  (f,g):
\begin{aligned}
& \;-g \in \Phi_b(\X); \\
&\; f(x)+\int g(y) \; p(dy) \leq \int C(x,y) \; p(dy),\ \forall x \in \mathbb{X}, p \in \Gamma(x)
\end{aligned}
\right\}
\end{equation*}
                                          
In particular, if we take $\Gamma(x)=\mathfrak{P}(\mathbb{X}), \; \forall x \in \mathbb{X}$, then it is easy to see that $(f,g) \in \mathcal{F}^{\Gamma}(\alpha,\beta)$ iff $f(x)+g(y) \leq C(x,y), \; \forall (x,y) \in \mathbb{X} \times \mathbb{X}$. In this case, we recover the classical duality result (see e.g. \cite[Theorem 1.42]{MR3409718}).
\end{corollary}

\section{Monotonicity principle}
In this section, we provide a monotonicity principle and an application. We again use the usual weak topology. The monotonicity principle is as follows.

\begin{theorem}\label{t5}
Let \asref{a1} hold. Assume $A\cap C$ is convex and weakly compact (or convex and weakly closed when $\Omega=\X^2$), and $\mathcal{T}_c^\Gamma(\alpha,\beta)$ defined in \eqref{eo1} is finite. Let $\pi^{*}$ be an optimizer of  $\mathcal{T}_c^\Gamma(\alpha,\beta)$. Then there exists a Borel set $\Lambda \subset \mathbb{X}$ with $\alpha(\Lambda)=1$, such that if $x,x' \in \Lambda$, $m_x \in \Gamma(x), m_{x'} \in \Gamma(x')$, and
\begin{equation}\label{eo3}
m_x+m_{x'}=\pi^{*}|_{X_0=x}+\pi^{*}|_{X_0={x'}},
\end{equation}
then
\begin{equation}\label{eo5}
c(x,\pi^{*}|_{X_0=x})+c(x',\pi^{*}|_{X_0={x'}}) \leq c(x,m_x)+c(x,m_{x'}).
\end{equation}
\end{theorem}

\begin{proof}
Take an optimizing sequence $\{\phi_n\}_n \in \Phi_b(\mathbb{X})$ for the right-hand-side of \eqref{eo2}. Note that 
$$\int \phi_n(y) \beta(dy)= \int_{x \in \mathbb{X}} \ \mathbb{E}^{\pi^{*}|_{X_0=x}}[\phi_n (X_T)] \ \alpha(dx).$$
We define 
\begin{equation}\label{eo4}
\begin{aligned}
&f_n(x):=R_c^{\Gamma} \phi_n(x)-\mathbb{E}^{\pi^{*}|_{X_0=x}}[\phi_n (X_T)] \\
=& \inf_{Q \in \Gamma(x)} \{ \mathbb{E}^Q[\phi_n(X_T)] + c(x,Q)-\mathbb{E}^{\pi^{*}|_{X_0=x}}[\phi_n (X_T)] \}. 
\end{aligned}
\end{equation}
Then it is clear that
$$ \int c(x,\pi^*_x) \alpha (dx) = \lim_{n \to \infty} \int f_n(x) \alpha(dx).$$
Since $\pi^{*}|_{X_0=x} \in \Gamma(x)$ $\alpha$- a.e, we have that $f_n(x) \leq c(x,\pi^{*}|_{X_0=x})$ by taking $p=\pi^{*}|_{X_0=x}$ on the right hand side of equation \eqref{eo4}. Because
$$\lim_{n \to \infty} \left(\int c(x,\pi^{*}|_{X_0=x})-f_n(x)\right) \,\alpha(dx)=0$$
and $c(x,\pi^{*}|_{X_0=x})-f_n(x) \geq 0$, we can find a Borel set $\Lambda \subset \mathbb{X}$ and a subsequence $f_{n(k)}$ such that $\alpha( \Lambda)=1$ and
$$\lim_{k \to \infty} f_{n_k}(x)=c(x,\pi^{*}|_{X_0=x})\quad\text{on}\ \Lambda.$$

It remains to show that $\Lambda$ has the monotonicity property. Let $x,x' \in \Lambda$ and $m_x \in \Gamma(x), m_{x'} \in \Gamma(x')$ satisfy \eqref{eo3}. By \eqref{eo4}, 
\begin{equation*}
\begin{split}
f_n(x)+f_n(x') &\leq \mathbb{E}^{m_x}[\phi_n(X_T)]+c(x,m_x)-\mathbb{E}^{\pi^*|_{X_0=x}}[\phi_n(X_T)] \\
                      &+ \mathbb{E}^{m_{x'}}[\phi_n(X_T)]+c(x',m_{x'}) -\mathbb{E}^{\pi^*|_{X_0=x'}}[\phi_n(X_T)] \\
                      &=c(x,m_x)+c(x',m_{x'}). \\
\end{split}
\end{equation*}
Then \eqref{eo5} follows by sending $n\to\infty$ 
\end{proof}
%

\begin{remark}
If we take $ \Omega =\mathbb{X}^2$ and $\Gamma(x)=\mathfrak{P}(\Omega^x), \; \forall x \in \mathbb{X}$, then our result recovers \cite[Proposition 4.1]{2017arXiv170804869B}. While we use Kantorovich duality in the proof, \cite{2017arXiv170804869B} uses a measurable selection argument.
\end{remark}

\subsection{Left-monotonicity when $\Omega=\mathbb{R}^2$}
In this part, we provide an application of \thref{t5}. It can be thought of as an extension of  \cite[Theorem 6.1]{MR3456332}.

Let $\Omega=\mathbb{R}^2$. Then $\Omega^x=\{x\} \times \mathbb{R}$ can be identified with $\mathbb{R}$, and $\{P|_{X_0=x}\}_{x \in \mathbb{R}}$ is the disintegration $\{P_x\}_{x \in \mathbb{R}}$. Let 
\begin{equation}\label{eq:constraint}
\Gamma(x)=\left\{Q \in \mathfrak{P}(\mathbb{R}):\  Q \{y: |y-x| \leq a(x)\}=1 , \; \int y \; Q (dy)=x \right\},
\end{equation}
where $a(\cdot)$ is a nonnegative, bounded and continuous function on $\mathbb{R}$.
\begin{definition}
A subset $\Delta \subset \mathbb{R}^2$ is called $\Gamma$-left monotone, if for every triple $(x,y^-),(x,y^+),(x',y') \in \Delta$ we cannot have the situation 
\begin{equation}
x< x', y^- < y '< y^+, |y'-x| \leq a(x), |y^- -x'| \leq a(x'), |y^+ -x'|\leq a(x').
\end{equation}
And a transport plan $\pi \in \mathfrak{P}(\mathbb{R}^2)$ is said to be $\Gamma$-left monotone if it concentrates on a $\Gamma$-left monotone set. 
\end{definition}

\begin{proposition}\label{prop:prop2}
Assume the cost function $c$ is given by
$$c(x,Q)=\int_{y\in \mathbb{R}}h(y-x)\,Q(dy),$$
where $h$ is a differentiable function on $\mathbb{R}$ with $h'$ strictly convex. Then any minimizer of the problem \eqref{eo1} is $\Gamma$-left monotone. 
\end{proposition}

\begin{proof}

By \prref{prop1}, $A \cap C$ is convex and weakly compact. Let $\pi^*$ be a minimizer of \eqref{eo1}. Let $\Lambda$ be given in \thref{t5}, and
$$\Delta=\cup_{x \in \Lambda}\{(x,y): y \in \supp(\pi^*_x)\}.$$
It is clear that $\pi^*(\Delta)=1$. Suppose there exists a triple $(x,y^-),(x,y^+),(x',y') \in \Delta$ violates $\Gamma$-left monotonicity. We strive for a contradiction.

Because
$$y^- < y' < y^+,\quad\{y^-,y^+\} \subset \supp(\pi^*_x),\quad y' \in \supp(\pi^*_{x'}),$$
we can construct two measures $\mu, \nu$ together with real numbers $l,r$ satisfying the following property: 
\begin{align}
\notag &\{y^-,y^+\} \subset \supp(\mu) \subset \{y: |y-x'| \leq a(x')\},\quad \mu \leq \pi^*_x;\\ 
\notag &y' \in \supp(\nu) \subset \{y: |y-x| \leq a(x)\},\quad \nu \leq \pi^*_{x'};\\
\label{eo8}&\text{$\mu$ and $\nu$ have the same barycenter and the same mass};\\
\label{eo9}&\text{$\mu$ is concentrated on $\mathbb{R} \setminus (l,r)$ while $\nu$ is concentrated on $[l,r]$}.
\end{align}
Let
$$m_x:=\pi^*_x -\mu +\nu\quad\text{and}\quad m_{x'}:=\pi^*_{x'}+\mu-\nu.$$
 It is clear that $m_x+m_{x'}=\pi^*_x+\pi^*_{x'}$ and $m_x \in \Gamma(x), m_{x'} \in \Gamma_{x'}$. Thanks to \eqref{eo8}, \eqref{eo9} and the strict convexity of $h'$, we can apply \cite[Example 2.4]{MR3456332} and get that
 $$\int h'(y-x) \; \mu(dy) > \int h'(y-x) \; \nu(dy).$$
 
Now we have 
\begin{align*}
   \int h(y-x) & \; \pi^*_{x}(dy) + \int h(y-x') \; \pi^*_{x'}(dy)  - \int h(y-x) \; m_x(dy)- \int h(y-x') \; m_{x'}(dy) \\
 =& \int h(y-x)\; (\mu-\nu)(dy)  -\int h(y-x') \; (\mu-\nu)(dy) \\  
=&\int_x^{x'} dz \int_{y \in \mathbb{R}} h'(y-z) (\mu-\nu) (dy) >0,
\end{align*}
which contradicts \eqref{eo5}.
\end{proof}

Here is an example such that $\Gamma$-left monotone transport plans may not be left monotone. \begin{example}
Take $\alpha=\frac{1}{2}(\delta_0+\delta_5), \beta=\frac{1}{4}(\delta_{-2}+\delta_0+\delta_2+\delta_{10}),$ and $$\Gamma(x)=\left\{Q \in \mathfrak{P}(\mathbb{R}): Q \{y: |y-x| \leq 6 \}=1, \int y \ Q(dy)=x \right\}.$$
It can be easily checked that $\frac{1}{4}(\delta_{(0,-2)}+\delta_{(0,2)}+\delta_{(5,0)}+\delta_{(5,10)} )$ is the unique $\Gamma$-left monotone transport plan, while $\frac{1}{8}(2\delta_{(0,0)}+\delta_{(0,-2)}+\delta_{(0,2)}+\delta_{(5,-2)}+\delta_{(5,2)}+2\delta_{(5,10)} )$ is the left-curtain coupling (i.e. the unique left monotone transport plan; see \cite{MR3456332}). 
\end{example}

Next, we will prove that the minimizer of the problem \eqref{eo1} is unique if the initial distribution $\alpha$ concentrates on two points.

\begin{proposition}
Under the assumption of \prref{prop:prop2}, if the initial distribution $\alpha$ concentrates on two points, then there exists at most one optimizer of problem \eqref{eo1}. 
\end{proposition}
\begin{proof}
Without loss of generality, we assume that $\alpha=p \delta_0 + (1-p) \delta_1$, where $p \in (0,1)$. Assuming that there are two optimizers $\pi$ and $\tilde{\pi}$, we prove the proposition by contradiction. Take $$A_0:=[-a(0),a(0)], \quad A_1:=[1-a(1),1+a(1)],$$
where $a(.)$ defines the constraint in \eqref{eq:constraint}. Define 
\begin{align*}
\beta_0=\beta|_{\supp(\beta)\setminus A_1}, \quad \beta_1=\beta|_{\supp(\beta) \setminus A_0}, \quad  \tilde{\beta}=\beta|_{A_0 \cap A_1}=\beta-\beta_0-\beta_1. 
\end{align*}
Note that the mass at initial position $0$ cannot be transported to $\supp(\beta) \setminus A_0$. Therefore the mass of $\beta_1$ must be transported from position $1$. Hence we have 
\begin{align*}
\pi_1|_{\supp{\beta} \setminus A_0}=\tilde{\pi}_1|_{\supp{\beta} \setminus A_0}=\beta_1/(1-p),
\end{align*}
 and similarly,
 \begin{align}\label{eq:firstmarginal}
\pi_0|_{\supp{\beta} \setminus A_1}=\tilde{\pi}_0|_{\supp{\beta} \setminus A_1}=\beta_0/p.
\end{align}

Since $\pi$ and $\tilde{\pi}$ are different, $\pi_0-\tilde{\pi}_0=\sigma^+-\sigma^-$ is a nontrivial signed measure with positive part $\sigma^+$ and negative part $\sigma^-$. Using \eqref{eq:firstmarginal}, the martingale condition, and the fact that $\pi_0(\mathbb{R})=\tilde{\pi}_0(\mathbb{R})=1$, we obtain that $\supp(\sigma^+) \cup \supp(\sigma^-) \subset A_0 \cap A_1$, and that
\begin{align}\label{eq:margmtg}
\int_{A_0 \cap A_1} x \ \sigma^+(dx)= \int_{A_0 \cap A_1} x \ \sigma^-(dx),  \quad \sigma^+(A_0 \cap A_1) =\sigma^-(A_0 \cap A_1).
\end{align}
Without loss of generality, assume that $y^+:=\max\{y: y \in \supp(\sigma^+)\} \geq \max \{y:y  \in \supp(\sigma^-)\}$, and that $\sigma^-(\{y^+\})=0$ if these two maximums are equal. Take $y^-:=\min\{y: y \in \supp(\sigma^+)\}$. As a result of \eqref{eq:margmtg}, there exists some $y' \in \supp(\sigma^-)$ such that $y^-< y'< y^+$. Therefore, we can find two positive measures $\mu, \nu$ together with two real numbers $l,r$ satisfying the following property: 
\begin{align*}
\notag &\{y^-,y^+\} \subset \supp(\mu) \subset \supp(\sigma^+),\quad \mu \leq \sigma^+;\\ 
\notag &y' \in \supp(\nu) \subset \supp(\sigma^-),\quad \nu \leq \sigma^-;\\
&\text{$\mu$ and $\nu$ have the same barycenter and the same mass};\\
&\text{$\mu$ is concentrated on $\mathbb{R} \setminus (l,r)$ while $\nu$ is concentrated on $[l,r]$}.
\end{align*}

Since $\pi$ and $\tilde{\pi}$ have the same terminal distribution, i.e., $p\pi_0+(1-p)\pi_1=p\tilde{\pi}_0+(1-p)\tilde{\pi}_1$, we can deduce that $\pi_1-\tilde{\pi}_1=\frac{p}{1-p}(\sigma^--\sigma^+)$, and hence $\frac{p}{1-p}\nu \leq \pi_1$. Construct a new coupling $\pi^*$ via $\pi^*_0=\pi_0-\mu+\nu$ and $\pi^*_1=\pi_1+\frac{p}{1-p}(\mu-\nu)$. Then by the same argument used in the last part of the proof of Proposition~\ref{prop:prop2}, it can be seen that 
\begin{align*}
p c(0, \pi^*_0)+(1-p)c(1,\pi^*_1)< p c(0, \pi_0)+(1-p)c(1,\pi_1),
\end{align*}
which contradicts our assumption that $\pi$ is an optimizer.
\end{proof}

\bibliographystyle{siam}
\bibliography{ref}

\appendix
\section{Proof of \prref{p1}}
\begin{proof}
``$\Longrightarrow$''. Take $P\in A\cap B\cap C$. For any $f\in\mathfrak{C}_H$, we have $\E^P|f(X_1)|<\infty$. Hence,
$$\beta(f)=\E^P[f(X_1)]=\E^P[\E^P[f(X_1)|X_0]]\geq\E^P[f^\Gamma(X_0)]=\alpha(f^\Gamma).$$

``$\Longleftarrow$''. Using a measurable selection argument, we can show that
\begin{equation}
\alpha(f^\Gamma)=\inf_{P\in A\cap C}\E^P[f],\quad \forall\,f\in \mathfrak{C}_H.
\end{equation}
Let
$$(A\cap C)\circ X_1^{-1}:=\{P\circ X_1^{-1}:\ P\in A\cap C\}.$$
Then $\beta$ is in the $H$-closure of $(A\cap C)\circ X_1^{-1} \cap \kP_H$, for otherwise by the Hahn-Banach  theorem (see e.g.\cite[Corollary 14.4]{MR0394084}) there would exist $f\in\mathfrak{C}_H$ such that
$$\beta(f)<\inf_{P\in A\cap C}\E^P[f(X_1)]=\alpha(f^\Gamma),$$
a contradiction.

Let $P_n\in A\cap C$ with $\beta_n:=P_n\circ X_1^{-1}$ such that $\beta_n\to\beta$ in the sense of \eqref{ee3}. It can be shown that the sequence $(P_n)$ is relatively $J$-compact (see \cite{Strassen}). Then there exists $P_\infty\in\kP_J$ such that up to a subsequence $P_n\to P_\infty$ in the sense of \eqref{ee3}. As $A\cap C$ is $J$-closed, $P_\infty\in A\cap C$. Moreover, $P_n\circ X_1^{-1}=\beta^n\to\beta$ implies that $P_\infty\in C$. The conclusion follows.
\end{proof}

\end{document}